\documentclass{amsart}
\usepackage[foot]{amsaddr}
\usepackage[utf8]{inputenc}
\usepackage{amsmath,amsthm}
\usepackage{color}

\usepackage[all]{xy}
\usepackage[normalem]{ulem}

\newtheorem{theorem}{Theorem}[section]
\newtheorem{lem}[theorem]{Lemma}
\newtheorem{lemma}[theorem]{Lemma}
\newtheorem{prop}[theorem]{Proposition}
\newtheorem{corol}[theorem]{Corollary}

\theoremstyle{definition}
\newtheorem{remark}[theorem]{Remark}
\newtheorem{definition}[theorem]{Definition}
\newtheorem{notation}[theorem]{Notation}

\numberwithin{equation}{section}

\title[Fullness of exceptional collections via stability conditions]{Fullness of exceptional collections via stability conditions -- A case study: the quadric threefold}

\author{Barbara Bolognese$^1$}
\address{$^1$Dipartimento di Matematica -- Università di Roma Tre}
\email{barbara.bolognese@uniroma3.it}

\author{Domenico Fiorenza$^2$}
\address{$^2$Dipartimento di Matematica -- Sapienza Universit\`a di Roma}
\email{fiorenza@mat.uniroma1.it}

\begin{document}
 
\begin{abstract}
A powerful tool of investigation of Fano varieties is provided by exceptional collections in their derived categories. Proving the fullness of such a collection is generally a nontrvial problem, usually solved on a case-by-case basis, with the aid of a deep understanding of the underlying geometry. Likewise, when an exceptional collection is not full, it is not straightforward to determine whether its ``residual'' category, i.e., its right orthogonal, is the derived category of a variety. We show how one can use the existence of Bridgeland stability condition these residual categories (when they exist) to address these problems. We examine  a simple case in detail: the quadric threefold $Q_3$ in $\mathbb{P}^{4}$.
We also give an indication how a variety of other classical results could be justified or re-discovered via this technique., e.g., the commutativity of the Kuznetsov component of the Fano threefold $Y_4$.  
\end{abstract}

\maketitle
\vspace{-0.5cm}
\section{Introduction}
Semiorthogonal decompositions, originally introduced by Bondal and Orlov\cite{bondal-orlov}, are one of the most insighful features triangulated categories can have. A classic example is the semiorthogonal decomposition produced via a full exceptional collection, i.e. via a collection of objects each of which generates a subcategory equivalent to the derived category of a point, interacting with each other with prescribed hom-vanishings and spanning the whole triangulated category. These exceptional objects behave like one-dimensional, simple generating blocks of their ambient category. Their existence is a specific feature of certain types of triangulated categories, notably the derived categories of some Fano varieties (see, e.g. \cite{Kuz08,Kuz18} and the references therein). 
Generally, when a triangulated category admits an exceptional collection, this collection is not full: it usually admits a semiorthogonal complement, its right orthogonal, sometimes called residual category or, when the triangulated category is in fact the derived category of a variety, Kuznetsov component (after \cite{Kuz15}). Kuznetsov components have been increasingly studied, as they somehow represent the non-trivial part of the derived categories they are embedded into, and have been seen and conjectured to encode subtle geometric information on their ambient varieties, notably (and conjecturally) their rationality properties \cite{Kuz08b}. 
\par
In the same flavor, Bridgeland stability conditions on the derived category of an algebraic variety $X$ are often used to investigate the geometric properties of moduli spaces of sheaves and complexes of sheaves on $X$, and their very existence is an intense object of study on threefolds and higher dimensional varieties \cite{Bri01,Bri03,BM12,BM13}. Recently, Bayer-Lahoz-Macrì-Stellari \cite{BLMS} have showed how to induce a stability condition on the Kuznetsov component of a projective variety from a weak stability condition on its hosting derived category, provided that certain  conditions are satisfied. These conditions are notably satisfied by Fano threefolds of Picard rank 1. 
\par
We illustrate how the existence of stability conditions on a Kuznetsov component automatically allows to prove certain results that, when already known, usually require case-by-case geometric techniques to be dealt with. Namely, 
showing that a given exceptional collection is full or, more generally, showing that a certain triangulated subcategory exhausts the  right orthogonal of an exceptional collection, are usually nontrivial problems due to the a priori possible presence of subcategories which are invisible to numerical detection: the so-called phantom subcategories \cite{gorchinskiy-orlov}. A stability condition, including a ``positivity condition'' for nonzero objects, notably forbids the presence of phantom subcategories (or at least, as we will show, of phantomic summands, which is enough for our purposes), thus remarkably simplifying proofs of fullness.
\par
In particular we will focus our attention on a simple example, while at the same time indicating how a few others can at least in principle be described in a similar fashion: the index 3 Fano threefold, i.e., a smooth quadric threefold $Q_3$ in $\mathbb{P}^4$. Quadrics are actually one of those few notable cases (among, e.g., projective spaces and Grassmannians) where exhibiting a full exceptional collection does not require somehow sophisticated techniques as stability conditions, see \cite{kapranov,sasha-quadrics}. We present $Q_3$ as a simple case study, to illustrate the effectiveness of the use of stability conditions (when they exist) in investigating fullness of exceptional collections. In particular,  
 we show how  in this case one rediscovers Kapranov's full exceptional collection $(S,\mathcal{O}_{Q_3},\mathcal{O}_{Q_3}(1),\mathcal{O}_{Q_3}(2))$, where $S$ is the spinor bundle over $Q_3$  \cite{kapranov}. 
 Among other possible applicatons, we sketch how one could recover the equivalence between the Kuznetsov component $\mathrm{Ku}(Y_4)$ of ${\mathcal{D}^b}(Y_4)$, where $Y_4$ is the index 2 Fano threefold  given by the complete intersection of two smooth generic quadric hypersurfaces in $\mathbb{P}^5$, and the derived category of the moduli space of spinor bundles on $Y_4$ \cite{bondal-orlov}, 
via the identification of this moduli space with a moduli space of Bridgeland-stable point-objects. Similarly, one sees how the the numerical condition $\chi(v,v)=0$ for a numerical classe $v$ of a cubic fourfold $W_4$ naturally shows up in ehibiting the equivalence between the Kuznetsov component $\mathrm{Ku}(W_4)$ and the derived category of a (possibly twisted) K3 surface \cite{addington-thomas,huybrechts,BLMNPS}.
\medskip

{\bf Acknowledgements.}
The authors wish to thank Arend Bayer and Alexander Kuznetsov, for useful comments on a first draft of this article.


\tableofcontents

\section{Notation and conventions}
We assume that the reader is familiar with the fundamental notions from the theory of stability conditions on triangulated categories and of semiorthogonal decompositions. See \cite{bondal-orlov,bayer,MS} for an introduction. In this section we will simply set up the necessary notation. 
\par
We will be working over the field $\mathbb{C}$ of complex numbers and we will denote the imaginary unit by $\sqrt{-1}$. By $(X,H)$ we will denote a primitively polarized, $n$-dimensional smooth projective variety $i\colon X\hookrightarrow \mathbb{P}^N$. For top dimensional cohomology classes of $X$, by a slight abuse of notation, we will simply write the class $\omega$ for the complex number $\int_X \omega$.
\par
We will denote the abelian category of coherent sheaves on $X$ by $\mathrm{Coh}(X)$, and its derived category by ${\mathcal{D}^b}(X)$. The Grothendieck and the numerical Grothendieck groups of ${\mathcal{D}^b}(X)$ will be denoted by $K(X)$ and by $K_{\mathrm{num}}(X)$, respetively. More generally, we will write $K(\mathcal{D})$ and $K_{\mathrm{num}}(\mathcal{D})$ to denote, respectively, the Grothendieck and numerical Grothendieck groups, of a numerically finite triangulated category $\mathcal{D}$, i.e., of a triangulated category $\mathcal{D}$ endowed with a Serre functor an such that for any two objects $E,F\in \mathcal{D}$ one has $\dim \mathrm{Hom}_{\mathcal{D}}(E,F[n])<+\infty$ for every $n\in \mathbb{Z}$ and $\dim \mathrm{Hom}_{\mathcal{D}}(E,F[n])=0$ for every $|n|>\!>0$.
\par
On ${\mathcal{D}^b}(X)$ we will consider the weak numerical stability condition $\sigma_H=(\mathrm{Coh}(X),\allowbreak Z_H)$, where $Z_H\colon K_{\mathrm{num}}(X)\to \mathbb{C}$ is defined by
\[
Z_H(E)=-H^{n-1}\mathrm{ch_1}(E)+\sqrt{-1}\,H^n\mathrm{ch}_0(E).
\]
For a nonzero object $E$ in $\mathrm{Coh}(X)$, the slope of $E$ with respect to the weak stability function $Z_H$ is
\[
\mu_H(E)=\frac{H^{n-1}\mathrm{ch_1}(E)}{H^n\mathrm{ch}_0(E)},
\]
where we set $\mu_H(E)=+\infty$ if $\mathrm{ch}_0(E)=0$.
\par
For any $\beta\in \mathbb{R}$, we denote by $\mathrm{Coh}_H^\beta(X)$ the heart of the $t$-structure on ${\mathcal{D}^b}(X)$ obtained by tilting the standard heart $\mathrm{Coh}(X)$ with respect to the torsion pair $(\mathrm{Coh}(X)_{\mu_H\leq \beta},\mathrm{Coh}(X)_{\mu_H> \beta})$, where
\[
\mathrm{Coh}(X)_{\mu_H\leq \beta}=\langle E\in \mathrm{Coh}(X)\colon\text{$E$ is $\sigma_H$-semistable with $\mu_H(E)\leq \beta$}\rangle
\]
\[
\mathrm{Coh}(X)_{\mu_H> \beta}=\langle E\in \mathrm{Coh}(X)\colon\text{$E$ is $\sigma_H$-semistable with $\mu_H(E)> \beta$}\rangle.
\]
For any $\alpha>0$, by $\sigma_{\alpha,\beta}=(\mathrm{Coh}_H^\beta(X),Z_{\alpha,\beta})$ we denote the weak stability condition on ${\mathcal{D}^b}(X)$ with heart $\mathrm{Coh}_H^{\beta}(X)$ and weak stability function
\begin{align}\label{eq:weak-stability-function}
Z_{\alpha, \beta}(E) &= -\int _X H^{n-2}\mathrm{ch}^{\beta+\sqrt{-1}\,\alpha}(E) 
\\
\notag &=\left(\frac{\alpha^2}{2}\mathrm{ch}_0^\beta(E)H^n\-\mathrm{ch}_2^\beta(E)H^{n-2}\right)+\sqrt{-1}\,\left(\alpha\mathrm{ch}_1^\beta(E) H^{n-1}\right)
\\
\notag &= \left(\frac{\alpha^2 - \beta^2}{2}\mathrm{ch}_0(E)H^n+ \beta \mathrm{ch}_1(E) H^{n-1}  -\mathrm{ch}_2(E) H^{n-2}\right)\\
\notag &\qquad\qquad\qquad+ \sqrt{-1}\,\left(-\alpha\beta \mathrm{ch}_0(E)H^n+\alpha \mathrm{ch}_1(E)H^{n-1}\right) ,
\end{align}
where as customary we write $\mathrm{ch}^{\gamma}(E)=e^{-\gamma H}\mathrm{ch}(E)$.
The associated slope, for a nonzero object $E$ in $\mathrm{Coh}_H^\beta(X)$, is
\[ 
\mu _{\alpha, \beta}(E)  = \frac{\mathrm{ch}_2(E)H^{n-2}- \beta \mathrm{ch}_1 H ^{n-1}+ \frac{\beta^2 - \alpha^2}{2}\mathrm{ch}_0(E)H^n}{\alpha \mathrm{ch}_1(E)H^{n-1}-\alpha\beta \mathrm{ch}_0(E)H^n} . \]
Finally, for any $\mu\in \mathbb{R}$, we denote by $\mathrm{Coh}_{\alpha,\beta}^\mu(X)$ the heart of the $t$-structure on ${\mathcal{D}^b}(X)$ obtained by tilting the heart $\mathrm{Coh}_H^\beta(X)$ with respect to the torsion pair $(\mathrm{Coh}^\beta(X)_{\mu_{\alpha,\beta}\leq \mu},\mathrm{Coh}^\beta(X)_{\mu_{\alpha,\beta}> \mu})$.
\begin{remark}\label{remark:four-cases}
For an object $E\in {\mathcal{D}^b}(X)$ which is at the same time $\sigma_H$-semistable and  $\sigma_{\alpha,\beta}$-semistable, the property of belonging to the doubly tilted heart $\mathrm{Coh}_{\alpha,\beta}^\mu(X)$ reduces to a pair of inequalities involving the slopes $\mu_{H}$ and $\mu_{\alpha,\beta}$, and the property of belonging to the standard heart $\mathrm{Coh}(X)$, of a suitable shift of $E$. Namely, for a $\sigma_H$- and $\sigma_{\alpha,\beta}$-semistable object $E$ in ${\mathcal{D}^b}(X)$ we have that $E\in \mathrm{Coh}_{\alpha,\beta}^\mu(X)$ precisely when one of the following four cases occurs:
\[
\begin{cases}
E\in \mathrm{Coh}(X), & \mu_H(E)>\beta, \quad \mu_{\alpha,\beta}(E)>\mu\\ \\
E[-1]\in \mathrm{Coh}(X),& \mu_H(E[-1])\leq \beta,\quad \mu_{\alpha,\beta}(E)>\mu\\ \\ 
E[-1]\in \mathrm{Coh}(X),& \mu_H(E[-1])> \beta,\quad \mu_{\alpha,\beta}(E[-1])\leq \mu\\ \\ 
E[-2]\in \mathrm{Coh}(X),& \mu_H(E[-2])\leq \beta,\quad \mu_{\alpha,\beta}(E[-1])\leq \mu\,.
\end{cases}
\]
\end{remark}

\begin{remark}\label{rem:in-codimension-3}
If the integral cohomology of $X$ is generated by $H$ in degree $\leq 4$, one sees that a nonzero object $E$ in $\mathrm{Coh}_{\alpha,\beta}^\mu(X)$ with $Z_{\alpha,\beta}(X)=0$ is a coherent sheaf on $X$ supported in codimension at least 3. This is easy and well known, but as we were not able to locate a completely explicit proof in the literature we provide it here for the reader's convenience. If $E\in \mathrm{Coh}_{\alpha,\beta}^\mu(X)$, then $E$ fits into a distinguished triangle
\[
E_{\leq \mu}[1]\to E\to E_{>\mu}\xrightarrow{+1} E_{\leq \mu}[2]
\]
in $\mathcal{D}^b(X)$, with $E_{> \mu}\in \mathrm{Coh}^\beta_H(X)_{\mu_{\alpha,\beta}>\mu}$ and $E_{\leq\mu}\in \mathrm{Coh}^\beta_H(X)_{\mu_{\alpha,\beta}\leq\mu}$. If $Z_{\alpha,\beta}(E)=0$, this gives $Z_{\alpha,\beta}(E_{\geq \mu})=Z_{\alpha,\beta}(E_{<0})=0$, as this is the only possible common value for $Z_{\alpha,\beta}(E_{\geq \mu})$ and $Z_{\alpha,\beta}(E_{<0})$. But $Z_{\alpha,\beta}(E_{\leq \mu})$ can not be $0$ unless $E_{\leq \mu}=0$, so $E$ is an object in $\mathrm{Coh}^\beta_H(X)$ with $Z_{\alpha,\beta}(E)=0$. Now consider the distinguished triangle
\[
F_{\leq \beta}[1]\to E\to F_{>\beta}\xrightarrow{+1} F_{\leq \beta}[2]
\]
in $\mathcal{D}^b(X)$, with $F_{> \beta}\in \mathrm{Coh}(X)_{\mu_{H}>\beta}$ and $F_{\leq\beta}\in \mathrm{Coh}(X)_{\mu_{H}\leq\beta}$. As $F_{> \beta},F_{\leq \beta}[1]$ in $\mathrm{Coh}^\beta(X)$, we have 
\[\mathrm{Im}(Z_{\alpha,\beta}(F_{>\beta}))\geq 0,\quad \mathrm{Im}(Z_{\alpha,\beta}(F_{\leq \beta}[1]))\geq 0.
\]
Additivity of $\mathrm{Im}(Z_{\alpha,\beta}))$ then gives $\mathrm{Im}(Z_{\alpha,\beta}(F_{>\beta}))=\mathrm{Im}(Z_{\alpha,\beta}(F_{\leq \beta}[1]))=0$, and so
\[\mathrm{Re}(Z_{\alpha,\beta}(F_{>\beta})\leq 0,\quad \mathrm{Re}(Z_{\alpha,\beta}(F_{\leq \beta}[1])\leq 0.
\]
Additivity again gives $\mathrm{Re}(\allowbreak Z_{\alpha,\beta}(F_{>\beta}))=\mathrm{Re}(Z_{\alpha,\beta}(F_{\leq \beta}[1]))=0$, and so $\mathrm{Re}(\allowbreak Z_{\alpha,\beta}(\allowbreak F_{\leq \beta})\allowbreak =0$.
Assume $F_{\leq \beta}$ is nonzero. Since $F_{\leq \beta}$ has finite slope, it must have positive rank. This implies that, if $F_{\leq \beta}$ has a slope-semistable factor with slope strictly less than $\beta$ we get $\mu_H(F_{\leq \beta})<\beta$ and so $\mathrm{Im}(Z_{\alpha,\beta}(F_{\leq \beta}))=\mathrm{ch}^\beta_1(F_{\leq \beta})<0$, a contradiction. Hence, $F_{\leq \beta}$ is actually a torsion-free slope semistable sheaf with slope $\mu_H(F_{\leq\beta})=\beta$. Using this we compute
\begin{align*}
  \mathrm{Re}(Z_{\alpha,\beta}(F_{\leq \beta}))&=\frac{\alpha^2(\mathrm{ch}_0(F_{\leq \beta})H^n)^2+\Delta_H(F_{\leq \beta})}{2(\mathrm{ch}_0(F_{\leq \beta})H^n)} 
  \\
  &\geq \frac{\alpha^2}{2}\mathrm{ch}_0(F_{\leq \beta})H^n>0,
\end{align*}
a contradiction. Here, for an object $F$ in $\mathcal{D}^b(X)$ we have written $\Delta_H(F)$ for the $H$-discriminant 
\[
\Delta_H(F)=(\mathrm{ch}_1(F)H^{n-1})^2-2(\mathrm{ch}_0(F)H^n)(\mathrm{ch}_2(F)H^{n-1})
\]
and we have used the Bogomolov-Gieseker-type inequality from \cite[Theorem 3.5]{BMS}: for $F$ a torsion-free slope semistable sheaf one has $\Delta_H(F)\geq 0$. The above contradiction shows that $E=F_{>\beta}$ and so $E$ is a coherent sheaf with $Z_{\alpha,\beta}(E)=0$. Looking at the real part of $Z_{\alpha,\beta}(E)$ we find
\[
\frac{\alpha^2 - \beta^2}{2}\mathrm{ch}_0(E)H^n+ \beta \mathrm{ch}_1(E) H^{n-1}  -\mathrm{ch}_2(E) H^{n-2}
\]
for any $\alpha>0$ and any $\beta\in \mathbb{R}$. This implies $\mathrm{ch}_{\leq 2}(E)=0$ and so, as $E$ is a coherent sheaf, that $E$ is supported in codimension at least 3.
\end{remark}

\section{The Kuznetsov component of the quadric threefold and its Serre functor}
Let $i\colon Q_3\hookrightarrow \mathbb{P}^4$ be a smooth quadric. The derived category of $\mathbb{P}^4$ has a semiorthogonal decomposition induced by a full exceptional collection
\[
{\mathcal{D}^b}(\mathbb{P}^4)=\langle \mathcal{O}_{\mathbb{P}^4}, \mathcal{O}_{\mathbb{P}^4}(1),\mathcal{O}_{\mathbb{P}^4}(2),\mathcal{O}_{\mathbb{P}^4}(3),\mathcal{O}_{\mathbb{P}^4}(4)\rangle,
\]
see \cite{beilinson}. The residual category $\mathrm{Ku}(Q_3)$ is defined as the right orthogonal to the exceptional collection $(\mathcal{O}_{Q_3},\mathcal{O}_{Q_3}(1),\mathcal{O}_{Q_3}(2))$ in ${\mathcal{D}^b}(Q_3)$.
See \cite{Kuz15} for details, where this residual category would be denoted as $\mathcal{A}_{Q_3}$. Here we use the  notation $\mathrm{Ku}(Q_3)$ as these residual categories are most commonly known as \emph{Kuznetsov components}. 
\begin{remark}\label{rem:admissibility}
If one has a semi-orthogonal decomposition $\mathcal{D}=\langle \mathcal{D}_1, \mathcal{D}_2\rangle$, then the subcategory $\mathcal{D}_1$ is \emph{left admissible}: the left adjoint $\iota^L_1$ to the inclusion functor $\iota _1: \mathcal{D}_1 \hookrightarrow \mathcal{D}$ is simply given by the projection functor $\tau_1\colon \mathcal{D}\to \mathcal{D}_1$ associated with the semiorthogonal decomposition. Remarkably, the converse is true: if $\mathcal{C}\hookrightarrow \mathcal{D}$ is a left admissible subcategory, then $\mathcal{C}$ is the left part of a semiorthogonal decomposition $\mathcal{D}=\langle \mathcal{C},{}^\perp\mathcal{C}\rangle$, where 
\[
{}^\perp\mathcal{C}=\{X\in \mathcal{D}\,"\, \mathrm{Hom}_{\mathcal{D}}(X,Y)=0, \quad \forall Y\in \mathcal{C}\}
\]
is the \emph{left-orthogonal} to $\mathcal{C}$; see \cite{bondal}. Dually, one has that a triangulated subcategory $\mathcal{C}\hookrightarrow \mathcal{D}$ is right admissible if and only if it is the right part of a semiorthogonal decomposition $\mathcal{D}=\langle \mathcal{C}^\perp,\mathcal{C}\rangle$. A subcategory $\mathcal{C}\hookrightarrow{\mathcal{D}}$ which is at the same time left and right admissible is simply called \emph{admissible}. In this case one has \emph{two} semiorthogonal decompositions $\mathcal{D}=\langle \mathcal{C},{}^\perp\mathcal{C}\rangle=\langle \mathcal{C}^\perp,\mathcal{C}\rangle$. Notice that one generally has ${}^\perp\mathcal{C}\neq {\mathcal{C}}^\perp$.

\end{remark}

\begin{remark}\label{rem:admissibility-2}

If the category $\mathcal{D}$ admits a Serre functor $\mathbb{S}$ and $\mathcal{D}_1\subseteq \mathcal{D}$ is left admissible, then also $\mathcal{D}_1$ has a Serre functor $\mathbb{S}_1$. More precisely, the inverse Serre functor on $\mathcal{D}_1$ is given by
\[
\mathbb{S}_{1}^{-1}=\tau_1\circ \mathbb{S}^{-1}.
\]
The existence of Serre functors both on $\mathcal{D}$ and $\mathcal{D}_1$ immediatley implies that $\mathcal{D}$ is also right admissible, and so admissible: the right adjoint $\iota^R$ to the inclusion functor $\iota _1:$ is given by the composition
\[ \iota_1^R = \mathbb{S}_1\circ \tau _1 \circ \mathbb{S}^{-1} \]
Dual considerations apply to the subcategory $\mathcal{D}_2$.

\end{remark}
\begin{remark}\label{rem:admissible-3}

Let $\mathcal{A}\subseteq \mathcal{B}\subseteq\mathcal{C}$ be an inclusion of triangulated subcategories. It is easy to see that if both $\mathcal{A}$ and $\mathcal{B}$ are admissible subcategories of $\mathcal{C}$, then $\mathcal{A}$ is an admissible subcategory of $\mathcal{B}$.

\end{remark}
\begin{remark}\label{rem:numerical-semiorthogonality}

Semiorthogonal decompositions have an immediate numerical counterpart. Assume
 $\mathcal{D}$ is a numerically finite triangulated category endowed with a Serre functor. Then a semiorthogonal decomposition $\mathcal{D}=\langle \mathcal{D}_1,\mathcal{D}_2\rangle$ induces a semiorthogonal decomposition 
\[
K_{\mathrm{num}}(\mathcal{D})=K_{\mathrm{num}}(\mathcal{D}_1)\oplus K_{\mathrm{num}}(\mathcal{D}_2).
\]
By this one means that one has a direct sum decomposition $K_{\mathrm{num}}(\mathcal{D})=K_{\mathrm{num}}(\mathcal{D}_1)\oplus K_{\mathrm{num}}(\mathcal{D}_2)$ of free $\mathbb{Z}$-modules, such that $K_{\mathrm{num}}(\mathcal{D}_1)$ is the right orthogonal to $K_{\mathrm{num}}(\mathcal{D}_2)$ with respect to the Euler pairing $\chi_{\mathcal{D}}$.
Moreover $(K_{\mathrm{num}}(\mathcal{D}_i),\chi_{\mathcal{D}_i})\hookrightarrow (K_{\mathrm{num}}(\mathcal{D}),\chi_{\mathcal{D}})$ are injective morphisms of free $\mathbb{Z}$-modules endowed with nondegenerate bilinear pairings.

\end{remark}

\begin{remark}
The semiorthogonal decomposition \[{\mathcal{D}^b}(Q_3)=\langle \mathrm{Ku}(Q_3),\mathcal{O}_{Q_3},\mathcal{O}_{Q_3}(1),\mathcal{O}_{Q_3}(2)\rangle\] is derived from the rectangular Lefschetz decomposition \[{\mathcal{D}^b}(\mathbb{P}^4)=\langle \mathcal{O}_{\mathbb{P}^4}, \mathcal{O}_{\mathbb{P}^4}(1),\mathcal{O}_{\mathbb{P}^4}(2),\mathcal{O}_{\mathbb{P}^4}(3),\mathcal{O}_{\mathbb{P}^4}(4)\rangle\] together with the spherical functor $i_*\colon {\mathcal{D}^b}(Q_3)\to {\mathcal{D}^b}(\mathbb{P}^4)$ induced by the divisorial embedding $i\colon Q_3\hookrightarrow \mathbb{P}^4$. Again, see \cite{Kuz15} for details. For later use, we explicitly note that the rectangular Lefschetz decomposition we are considering on ${\mathcal{D}^b}(\mathbb{P}^4)$ has lenght $m=5$ and that in going from the semiorthogonal decomposition of ${\mathcal{D}^b}(\mathbb{P}^4)$ to that of ${\mathcal{D}^b}(Q_3)$ we `lose' $d=2$ terms from the exceptional collection.  
\end{remark}

Following \cite{Kuz15} (see also the exposition in \cite{MS18}) one can easily determine a more explicit expression for the Serre functor $\mathbb{S}_{{\mathrm{Ku}}(Q_3)}$ with respect to the somehow implicit one given in Remark \ref{rem:admissibility}. We will use it in Section \ref{section:point-objects} to show that the Serre functor $\mathbb{S}_{\mathrm{Ku}(Q_3)}$ acts as the identity functor on the spinor bundle of $Q_3$.

\begin{notation}
We write 
\[
\mathrm{O}_{\mathrm{Ku}(Q_3)}\colon \mathrm{Ku}(Q_3)\to \mathrm{Ku}(Q_3)
\]
for the \emph{rotation} (or \emph{degree shift}) endofunctor of $\mathrm{Ku}(Q_3)$ given by
\[
\mathrm{O}_{\mathrm{Ku}(Q_3)}\colon E\mapsto \tau_{\mathrm{Ku}(Q_3)}(E(1)).
\]
\end{notation}
In our situation, $m=5$ an $d=2$, and so $m-d=3$; therefore \cite[Lemma 2.40]{MS18} gives that for any $0\leq n\leq 3$ one has $\mathrm{O}_{\mathrm{Ku}(Q_3)}^n(E)=\tau_{\mathrm{Ku}(Q_3)}(E(n))$. As an immediate consequence we get the following particular case of \cite[Example 2.2]{kuznetsov-smirnov}.

\begin{lemma}\label{lemma:inverse-serre} The Serre functor on the residual category $\mathrm{Ku}(Q_3)$ is given by
\[
\mathbb{S}_{\mathrm{Ku}(Q_3)}=\mathrm{O}^{-3}_{\mathrm{Ku}(Q_3)}[3],
\]
\end{lemma}
\begin{proof}
The inverse Serre functor on ${\mathcal{D}^b}(Q_3)$ is $\mathbb{S}_{Q_3}^{-1}(E)=E(3)[-3]$ and so, by Remark \ref{rem:admissibility}, $\mathbb{S}_{\mathrm{Ku}(Q_3)}^{-1}(E)=\tau_{\mathrm{Ku}(Q_3)}(E(3))[-3]=\mathrm{O}_{\mathrm{Ku}(Q_3)}^3(E)[-3]$. 
\end{proof}

\section{Point objects and numerical point objects}
We now describe the numerical Grothendieck group of the Kuznetsov component $\mathrm{Ku}(Q_3)$ and determine an even codimensional numerical point object, i.e., a primitive eigenvector for the numerical action of the Serre functor on it. We will see in Section \ref{section:point-objects} that this numerical point object comes from an actual point object.

\bigskip

If $\mathcal{A}$ is a numerically finite triangulated category with a Serre functor $\mathbb{S}$, then $\mathbb{S}$ induces an isometry of $\mathbb{Z}$-modules endowed with bilinear pairings
\[
\mathbb{S}_{\mathrm{num}}\colon (K_{\mathrm{num}}(\mathcal{A}),\chi_{\mathcal{A}})\to (K_{\mathrm{num}}(\mathcal{A}),\chi_{\mathcal{A}}).
\]

\begin{definition}
Let $\mathcal{A}$ be a numerically finite triangulated category with a Serre functor $\mathbb{S}$. Let $[d]\in \mathbb{Z}/2\mathbb{Z}$. A numerical class $[E]$ in $K_{\mathrm{num}}(\mathcal{A})$ is called a codimension $[d]$ \emph{numerical point object} if
\[
\mathbb{S}_{\mathrm{num}}([E])=(-1)^{[d]}[E].
\]
\end{definition}
The definition of numerical point object is motivated by Bondal-Orlov's definition of point object in a triangulated category with a Serre funtor, which we recall below.
\begin{definition}[\cite{bondal-orlov}, Definition 4.1]\label{def:BO-point-objects}
Let $\mathcal{A}$ be a triangulated category with a Serre functor $\mathbb{S}$, and let $d\in \mathbb{Z}$. An object $E$ in $\mathcal{A}$ is called a {\it codimension $d$ point object} if
\begin{enumerate}
    \item $\mathrm{Hom}_{\mathcal{A}}(E,E)=\mathbb{C}$;
    \item $\mathrm{Hom}_{\mathcal{A}}(E,E[n])=0$ for every $n<0$;
    \item $\mathbb{S}_{\mathcal{A}}(E)\cong E[d]$.
\end{enumerate}
\end{definition}
Since the shift functor acts as the multiplication by $-1$ on $K_{\mathrm{num}}(\mathcal{A})$, one immediately sees that if $\mathcal{A}$ is numerically finite and $E$ is a codimension $d$ point object in $\mathcal{A}$, then the numerical class $[E]$ is a codimension $[d]$ numerical point object in $K_{\mathrm{num}}(\mathcal{A})$.

\bigskip
With these premises, we can now look for numerical point ojects in the  Kuznetsov component $\mathrm{Ku}(Q_3)$. We will see in the next section that these numerical point objects come from actual point objects.
The semiorthogonal decomposition ${\mathcal{D}^b}(Q_3)=\langle \mathrm{Ku}(Q_3),\mathcal{O}_{Q_3},\mathcal{O}_{Q_3}(1),\mathcal{O}_{Q_3}(2)\rangle$ induces a semiorthogonal decomposition at the level of numerical Grothendieck groups, with respect to the (non-symmetric) Euler pairing. Thus $K_{\mathrm{num}}(\mathrm{Ku}(Q_3))$ is naturally identified with the right orthogonal to the three Chern vectors $\mathrm{ch}(\mathcal{O_{Q_3}})=1$, $\mathrm{ch}(\mathcal{O}_{Q_3}(1))=e^H$ and $\mathrm{ch}(\mathcal{O}_{Q_3}(2))=e^{2H}$ in 
\[
K_{\mathrm{num}}(Q_3)\xrightarrow[\raisebox{4pt}{$\sim$}]{\mathrm{ch}} \mathbb{Z}\oplus \mathbb{Z}H\oplus \mathbb{Z}\frac{H^2}{2}\oplus \mathbb{Z}\frac{H^3}{12}\cong\mathbb{Z}\oplus \mathbb{Z}\oplus \mathbb{Z}\frac{1}{2}\oplus \mathbb{Z}\frac{1}{12} 
\]
(see \cite{Fritzsche,Kuz08}), with respect to the pairing induced by the Euler form $\chi$ on ${\mathcal{D}^b}(Q_3)$. The Todd class of $Q_3$ is 
\[
\mathrm{td}_{Q_3}=1+\frac{3}{2}H+\frac{13}{12}H^2+\frac{1}{2}H^3,
\]
hence we see that the matrix representing the Euler pairing with respect to the  $\mathbb{Q}$-basis $\{1,H,H^2,H^3\}$ of $K_{\mathrm{num}}(Q_3)\otimes \mathbb{Q}$
is
\[
\left(
\begin{matrix}
      1/2 & 13/12  &  3/2  &    1\\
-13/12 &  -3/2  &   -1  &    0\\
   3/2  &    1  &    0  &    0\\
    -1  &    0  &    0  &    0
\end{matrix}
\right)
\]
The right orthogonal to $1,e^H,e^{2H}$ is therefore determined by the equation
\[
\left(
\begin{matrix}
      1 & 0  &  0  &    0\\
1 &  1  &   1/2  &    1/6\\
1  & 2  &    2  &    4/3
\end{matrix}
\right)
\left(
\begin{matrix}
      1/2 & 13/12  &  3/2  &    1\\
-13/12 &  -3/2  &   -1  &    0\\
   3/2  &    1  &    0  &    0\\
    -1  &    0  &    0  &    0
\end{matrix}
\right)
\left(\begin{matrix}
a_0\\
a_1\\
a_2\\
a_3
\end{matrix}
\right)
=0
\]
from which one immediately sees that a solution $(a_0,a_1,a_2,a_3)$ in $K_{\mathrm{num}}(Q_3)$ must be an integer multiple of the primitive lattice vector $(2,-1,0,\frac{1}{12})$. In other words, if $K_{\mathrm{num}}(\mathrm{Ku}(Q_3))$ is generated by the Chern character of an object $S$, one must have $\mathrm{ch}(S)=2-H+\frac{1}{12}H^3$. As we are going to see, this is precisely the Chern character of the spinor bundle on $Q_3$.
\begin{lemma}
The numerical action of $\mathbb{S}_{\mathrm{Ku}(Q_3)}$ on $K_{\mathrm{num}}(Q_3)$ is the identity. In particular, via the isomorphism $K_{\mathrm{num}}(Q_3)\xrightarrow[\raisebox{4pt}{$\sim$}]{\mathrm{ch}} \mathbb{Z}\oplus \mathbb{Z}H\oplus \mathbb{Z}\frac{H^2}{2}\oplus \mathbb{Z}\frac{H^3}{12}$ the lattice vector $2-H+\frac{1}{12}H^3$ is an even dimensional numerical point object.
\end{lemma}
\begin{proof}
Since $K_{\mathrm{num}}(\mathrm{Ku}(Q_3))\cong \mathbb{Z}$, the numerical action of the Serre functor has necessarily to be the identity or minus the identity. From the identity
\[
\chi^{}_{\mathrm{Ku}(Q_3)}(v,\mathbb{S}_{\mathrm{Ku}(Q_3);\mathrm{num}}v)=\chi^{}_{\mathrm{Ku}(Q_3)}(v,v)
\]
and the nondegeneracy of $\chi_{\mathrm{Ku}(Q_3)}$ we see that $\mathbb{S}_{\mathrm{Ku}(Q_3);\mathrm{num}}=\mathrm{id}_{K_{\mathrm{num}}(\mathrm{Ku}(Q_3))}.$
\end{proof}
\begin{remark}
 In higher rank situations,  the existence of numerical point objects in the Kuznetsov component is a nontrivial requirement. For instance, for index 2 Fano threefolds with Picard rank 1 this requirement singles out $Y_4$ as the only case where odd codimensional numerical point objects exist, and the double covering $Y_2$ of $\mathbb{P}^3$ ramified over a quadric as only case with even codimensional numerical point objects. Not surprisingly, the numerical point object for $Y_4$ is the Chern character $2-1H+\frac{1}{12}H^3$ of the spinor bundles. On the other hand, every primitive vector in $K_{\mathrm{num}}(\mathrm{Ku}(Y_2))$ is an even codimensional point object. Indeed, the Serre 
functor of $\mathrm{Ku}(Y_2)$ is the composition of the shift by 2 with the involution $\tau$ of $Y_2$ (see \cite[Corollary 4.6]{Kuz15}) and the generator of the Picard group $\mathrm{Pic}(Y_2)$ is pulled back from $\mathbb{P}^3$, and so is $\tau$-invariant (see, e.g., \cite{welters}).
\end{remark}

\section{Spinor bundles as point objects}\label{section:point-objects}
On an odd dimensional smooth quadric hypersurface $Q_{2m+1}\hookrightarrow \mathbb{P}^{2m+2}$ one has a distingushed rank $2^m$ vector bundle, induced by the spinor representation of $\mathfrak{so}(2m+3;\mathbb{C})$ associated with the quadratic form defining the quadric. This vector bundle is called the \emph{spinor bundle} and will be denoted as $S$. Similarly,
on an even dimensional smooth quadric $Q_{2m}\hookrightarrow \mathbb{P}^{2m+1}$ one has two nonisomorphic rank $2^{m-1}$ spinor bundles $S^+,S^-$, induced by the two half-spin representations of $\mathfrak{so}(2m;\mathbb{C})$. See \cite{ottaviani,sasha-perry} for details. In what follows we will make use of various elementary properties of spinor bundles on odd quadrics. We point the reader towards \cite{ottaviani} for complete statements and proofs.

\begin{lemma}[\cite{ottaviani}, Remark 2.9]
The Chern character of the spinor bundle $S$ on the quadric threefold $Q_3$ is 
\[
\mathrm{ch}(S)=2-H+\frac{1}{12}H^3.
\]
In particular, we have $K_{\mathrm{num}}(\mathrm{Ku}(Q_3))=\mathbb{Z}[S]$, where $[S]$ is the numerical class of $S$.
\end{lemma}

%
%

\begin{lemma}\label{lemma:S-in-Ku}
The spinor bundle $S$ is an object of the Kuznetsov component $\mathrm{Ku}(Q_3)$
\end{lemma}
\begin{proof}
We have to show that $\mathrm{Hom}_{{\mathcal{D}^b}(Q_3)}(\mathcal{O}(i),S[n])=0$, for every $i\in\{0,1,2\}$ and every $n\in \mathbb{Z}$. As $\mathcal{O}(i)$ and $S$ are locally free sheaves on $Q_3$, this is equivalent to
\[
H^n(Q_3,S(-i))=0
\]
for $i\in\{0,1,2\}$ and every $n\in \mathbb{Z}$. This is trivial for $n<0$ and for $n>3$, while for $0\leq n<3$ it is a particular case of \cite[Theorem 2.3]{ottaviani}. For $n=3$ we argue by Serre duality:
\begin{align*}
H^3(Q_3,S(-i))&=H^0(Q_3;S^*(i-3))^\vee\\
&=H^0(Q_3,S^*(i-3))^\vee\\
&=H^0(Q_3,S^*(i-3))^\vee\\
&=H^0(Q_3,S(i-2))^\vee\\
&=0,
\end{align*}
where we use the isomorphism $S^\ast\cong S(1)$ (see \cite[Theorem 2.8]{ottaviani}) and \cite[Theorem 2.3]{ottaviani} again.
\end{proof}
\begin{lemma}\label{lemma:S-as-point-object}
The spinor bundle $S$ on $Q_3$ is a 0-codimensional point object in ${\mathrm{Ku}}(Q_3)$ in the sense of Bondal-Orlov \cite{bondal-orlov}.
\end{lemma}
\begin{proof}
We have to show that the three conditions from Definition \ref{def:BO-point-objects} are satisfied.
The statements appearing in the first two conditions for the spinor bundle are classical: since $S$ is a locally free sheaf, we have
\begin{align*}
    \mathrm{Hom}_{{\mathrm{Ku}}(Q_3)}(S,S[n])&=\mathrm{Hom}_{{\mathcal{D}^b}(Q_3)}(S,S[n])\\
    &=\mathrm{Ext}^n_{\mathrm{Coh}(Q_3)}(S,S)\\
    &=H^n(Q_3,S^\ast\otimes_{\mathcal{O}_{Q_3}}S).
\end{align*}
This immediately gives $\mathrm{Hom}_{{\mathrm{Ku}}(Q_3)}(S,S[n])=0$ for every $n<0$, while for $n=0$ it is known that $H^0(Q_3,S^\ast\otimes_{\mathcal{O}_{Q_3}}S)=\mathbb{C}$, see, e.g., \cite[Lemma 2.7]{ottaviani}.
As far as the third condition is concerned, by rotating the short exact sequence of vector bundles
\[ 0 \to S \to \mathcal{O}^4_{Q_3} \to S(1) \to 0 \]
\cite[Theorem 2.8]{ottaviani} we obtain the distinguished triangle
\[ \mathcal{O}^4_{Q_3} \to S(1) \to S[1] \overset{+1}{\to} \mathcal{O}^4_{Q_3}[1]. \]
Here $\mathcal{O}_{Q_3}$ is an object in the triangulated subcategory of ${\mathcal{D}^b}(Q_3)$ spanned by the eceptional collection $(\mathcal{O}_{Q_3},\mathcal{O}_{Q_3}(1),\mathcal{O}_{Q_3}(2))$, while $S[1]$ is an object in the Kuznetsov component ${\mathrm{Ku}}(Q_3)$. Hence the image of $S(1)$ in ${\mathrm{Ku}}(Q_3)$ via the truncation functor $\tau_{\mathrm{Ku}(Q_3)}\colon {\mathcal{D}^b}(Q_3)\to {\mathrm{Ku}}(Q_3)$ is $S[1]$. As the composition $\tau_{\mathrm{Ku}(Q_3)}\circ (\mathcal{O}_{Q_3} (1)\otimes -)$ is the rotation functor ${\mathrm{O}}_{\mathrm{Ku}(Q_3)}$ for the Kuznetsov component, we obtain
$\mathrm{O}_{\mathrm{Ku}(Q_3)} (S) = S[1]$. Therefore, Lemma \ref{lemma:inverse-serre} gives  
%
%
%
\[ \mathbb{S}_{\mathrm{Ku}(Q_3)} (S)= O_{\mathrm{Ku}}^{-3}(S)[3] = S. \]
\end{proof}
The following result is classical. We reprove by using the fact $S$ is a point object in the Kuznetsov component.
\begin{corol}
The spinor bundle $S$ is an exceptional object in ${\mathcal{D}^b}(Q_3)$ and so $(S,\mathcal{O}_{Q_3},\mathcal{O}_{Q_3}(1),\mathcal{O}_{Q_3}(2))$ is an exceptional collection. 
\end{corol}
\begin{proof}
We need to show that 
\[
\mathrm{Hom}_{{\mathcal{D}^b}(Q_3)}(S,S[n])\cong\begin{cases}
\mathbb{C}\qquad&\text{if $n=0$}\\
0\qquad&\text{if $n=0$}
\end{cases}
\]
As $\mathrm{Ku}(Q_3)$ is a full subcategory od ${\mathcal{D}^b}(Q_3)$, by Lemma \ref{lemma:S-as-point-object} we only need to prove that $\mathrm{Hom}_{\mathrm{Ku}(Q_3)}(S,S[n])=0$ for $n>0$. By Lemma \ref{lemma:S-as-point-object} again, we have
\begin{align*}
\mathrm{Hom}_{\mathrm{Ku}(Q_3)}(S,S[n])&=\mathrm{Hom}_{\mathrm{Ku}(Q_3)}(S,\mathbb{S}_{\mathrm{Ku}(Q_3)}S[n])\\
&=\mathrm{Hom}_{\mathrm{Ku}(Q_3)}(S,S[-n])^\vee\\
&=0,
\end{align*}
for $n>0$.
\end{proof}

In order to exhibit a numerical stability condition on ${\mathrm{Ku}}(Q_3)$ we will make use of the following result.
\begin{prop}[Proposition 5.1.\cite{BLMS}]\label{prop:induces}
Let $(E_1,\dots,E_m)$ be an exceptional collection in a triangulated category $\mathcal{D}$, and let $\mathcal{K}=\langle\{E_i\}\rangle^\perp\subseteq \mathcal{D}$ be the corresponding  right orthogonal. 
Let $\sigma=(\mathcal{A},Z)$ be a weak stability condition on $\mathcal{D}$, and let $\mathcal{A}_{\mathcal{K}}=\mathcal{A}\cap \mathcal{K}$. Assume $\mathcal{D}$ has a Serre functor $S$ and that the following hold:
\begin{enumerate}
    \item $E_i, S(E_i)[-1]\in \mathcal{A}$ for every $i\in\{1,\dots,m\}$;
    \item $Z(E_i)\neq 0$ for every $i\in\{1,\dots,m\}$;
    \item for any nonzero object $F\in \mathcal{A}_{\mathcal{K}}$ one has $Z(F)\neq 0$.
\end{enumerate}
Then $(Z\bigr\vert_{\mathcal{A}_{\mathcal{K}}},\mathcal{A}_{\mathcal{K}})$ is a stability condition on $\mathcal{K}$.
\end{prop}
We can now prove
\begin{prop}\label{prop:stability-condition-on-ku}
Let $0<\alpha<\frac{1}{2}$. Then the weak stability condition $(Z_{\alpha,-\frac{1}{2}},\allowbreak \mathrm{Coh}^0_{\alpha,-\frac{1}{2}}(Q_3))$ induces a numerical stability condition on $\mathrm{Ku}(Q_3)$, whose heart is $\mathrm{Ku}(Q_3)\cap \mathrm{Coh}^0_{\alpha,-\frac{1}{2}}(Q_3)$.  
\end{prop}
\begin{proof}
We check that conditions (1)-(3) in proposition \ref{prop:induces} are satisfied. As the weak stability condition $(Z_{\alpha,-\frac{1}{2}},\mathrm{Coh}^0_{\alpha,-\frac{1}{2}}(Q_3))$ is numerical, so is the induced stability condition on $\mathrm{Ku}(Q_3)$.
The exceptional collection of $\mathcal{D}^b(Q_3)$ defining $\mathrm{Ku}(Q_3)$ as its right orthogonal is $(\mathcal{O}_{Q_3},\mathcal{O}_{Q_3}(1),\mathcal{O}_{Q_3}(2))$, and the Serre functor of $\mathcal{D}^b(Q_3)$ is $\mathbb{S}_{Q_3}(E)=E(-3)[3]$. So, we need to check that
\begin{equation}\label{eq:in-the-tilted-heart}
\mathcal{O}_{Q_3}(n), \mathcal{O}_{Q_3}(n-3)[2]\in \mathrm{Coh}^0_{\alpha,\frac{1}{2}}(Q_3), \qquad \forall n\in \{0,1,2\}.
\end{equation}
For any $n\in \mathbb{Z}$, the line bundle $\mathcal{O}_{Q_3}(n)$ is slope stable and satisfies $\Delta_H(\mathcal{O}_{Q_3}(n))=0$, where
\[
\Delta_H(E)=(\mathrm{ch}_1(E)H^2)^2-2(\mathrm{ch}_0(E)H^3)(\mathrm{ch}_2(E)H).
\]
So from \cite[Proposition 2.14]{BLMS} we have that $\mathcal{O}(n)[k]$ is both $\sigma_H$- and $\sigma_{\alpha,-\frac{1}{2}}$-semistable for every $n,k\in \mathbb{Z}$. As $\mathcal{O}(n)[k]\in \mathrm{Coh}(Q_3)[k]$, joint $\sigma_H$- and  $\sigma_{\alpha,-\frac{1}{2}}$-semistability reduces checking (\ref{eq:in-the-tilted-heart}) to checking the condition on the slopes from Remark \ref{remark:four-cases}.
As
\[
Z_H(\mathcal{O}(n)[k])=-2n+2\sqrt{-1}\,
\]
and
\[
Z_{\alpha,-\frac{1}{2}}(\mathcal{O}(n)[k])=(-1)^k\left((\alpha^2 -n^2-n-\frac{1}{4})+\sqrt{-1}\,\alpha(2n+1) \right)
\]
%
we see that for any $n\in \{0,1,2\}$ and any $0<\alpha<\frac{1}{2}$ we have $\mathcal{O}(n)\in \mathrm{Coh}(Q_3)$ with
\[
\mu_{H}(\mathcal{O}(n))=n>-\frac{1}{2},\quad \mu_{\alpha,-\frac{1}{2}}(\mathcal{O}(n))=\frac{n^2 +n+\frac{1}{4}-\alpha^2}{\alpha(2n+1)} >0,
\]
and $(\mathcal{O}(n-3)[2])[-2]\in\mathrm{Coh}(Q_3)$ with
\[
 \mu_{H}((\mathcal{O}(n-3)[2])[-2])=n-3\leq -\frac{1}{2}\]
 and
 \[
 \mu _{\alpha, -\frac{1}{2}}(\mathcal{O}(n-3)[2][-1]) =\frac{n^2-5n+\frac{25}{4}-\alpha^2}{\alpha(2n-5)}\leq 0.
\]
 Next, it is immediate that $Z_{\alpha,-\frac{1}{2}}(\mathcal{O}(n))\neq 0$ and $Z_{\alpha,-\frac{1}{2}}(\mathcal{O}(n-3)[2])\neq 0$ for any $n\in \{0,1,2\}$.
Finally, to show that for any nonzero object $F\in \mathrm{Coh}^0_{\alpha,\frac{1}{2}}(Q_3)\cap \mathrm{Ku}(Q_3)$ we have $Z_{\alpha,-\frac{1}{2}}(F)\neq 0$, recall from Remark \ref{rem:in-codimension-3} that a nonzero object $F\in \mathrm{Coh}^0_{\alpha,\frac{1}{2}}(Q_3)$ with $Z_{\alpha,-\frac{1}{2}}(F)= 0$ is a coherent sheaf supported in dimension 0. This implies 
\[
\mathrm{Hom}_{{\mathcal{D}^b}(Q_3)}(\mathcal{O},F)=\mathrm{Hom}_{\mathrm{Coh}(Q_3)}(\mathcal{O},F)\neq 0,
\]
and so $F$ can not be an object in $\mathrm{Ku}(Q_3)$.
\end{proof}

Next, we show that the spinor bundle $S$ is $\sigma_{\alpha,-\frac{1}{2}}$-semistable for the values of $\alpha$ as in the statement of Proposition \ref{prop:stability-condition-on-ku}. To do this, we will use the explicit wall and chamber structure of the $(\alpha,\beta)$-half-plane, see \cite{maciocia,Barbara}. 

\begin{lem} {\bf (No-wall lemma)}\label{nowall}
The $(\alpha,\beta)$-half-plane contains no walls for the truncated Chern vector $\mathrm{ch}_{\leq 2}(S)=2-H$ of the spinor bundle $S$ on $Q_3$.
\end{lem}

\begin{proof}
Let $v=\mathrm{ch}_0+\mathrm{ch}_1+\mathrm{ch}_2$ be a truncated vector in the Mukai lattice. By \cite[Corollary 2.8]{maciocia}, if $\mathrm{ch}_0H^3>0$ and $F>0$, where
\[
F=\frac{(\mathrm{ch}_1H^2)^2-2(\mathrm{ch}_0H^3)(\mathrm{ch}_2H)}{(\mathrm{ch}_0H^3)^2},
\]
then every numerical wall for $v$ intersects the vertical line $\beta=\beta_0$, where
\[
\beta_0=\frac{\mathrm{ch}_1 H^2}{\mathrm{ch}_0 H^3}-\sqrt{F}.
\]
If an actual wall exists, this must intersect this vertical line at some point $(\alpha_0,\beta_0)$ and so, by definition of wall, there exists a pair $(E,E_0)$ of $\sigma_{\alpha_0,\beta_0}$-semistable objects in $\mathrm{Coh}^{\beta_0}(X)$ with $\mathrm{ch}_{\leq 2}(E)=v$, and $E_0$ a subobject of $E$ with $\mu_{\alpha_0,\beta_0}(E_0)=\mu_{\alpha_0,\beta_0}(E)$. As 
\[
\mathrm{ch}_1^{\beta_0}(E)H^2=\mathrm{ch}_1H^2-\beta_0\mathrm{ch}_0H^3 =\sqrt{F} \mathrm{ch}_0H^3>0,
\]
an $\alpha_0$, this gives $\mathrm{Im}(Z_{\alpha_0,\beta_0}(E))>0$. The destabilizing short exact sequence for $(E,E_0)$ and the additivity of $Z_{\alpha_0,\beta_0}$ then imply
\[
0<\mathrm{ch}_1^{\beta_0}(E_0)H^2<\mathrm{ch}_1^{\beta_0}(E)H^2=\sqrt{F} \mathrm{ch}_0H^3.
\]
The truncated Chern character for the spinor bundle $S$ is $2-H$. So $\beta_0=-1$ and $\sqrt{F} \mathrm{ch}_0H^3=2$ and for any potentially destabilizing subobject $E_0$ we get
\[
0<\mathrm{ch}_1^{-1}(E_0)H^2<2.
\]
As $\beta_0=-1$ is an integer, $\mathrm{ch}_1^{-1}(E_0)$ is an element of $H^2(X;\mathbb{Z})=\mathbb{Z}H$, so we have $\mathrm{ch}_1^{-1}(E_0)H=n_0H$ for some integer $n_0$. Therefore, we obtain
\[
0<2n_0<2
\]
which is clearly impossible. So no actual wall can exist for $\mathrm{ch}_{\leq 2}(S)$.
\end{proof}
\begin{corol}
The spinor bundle $S$ on $Q_3$ is a Bridgeland semistable object in $\mathrm{Ku}(Q_3)$ with respect to the stability function $Z_{\alpha,-\frac{1}{2}}$, for any $\alpha>0$. Moreover, $S[1]$ lies in the heart $\mathrm{Coh}_{\alpha,-\frac{1}{2}}^0(Q_3)\cap \mathrm{Ku}(Q_3)$ of $\mathrm{Ku}(Q_3)$ with respect to the stability function $Z_{\alpha,-\frac{1}{2}}$.
\end{corol}
\begin{proof}
The spinor bundle $S$ is slope stable (see, e.g, \cite{ottaviani}), so it \sout{is} 
is $\sigma_{\alpha, \beta}$-stable for $\alpha$ sufficiently big and any $\beta$ by \cite[Proposition 2.13]{BLMS}. Since there are no walls in the $(\alpha,\beta)$-semiplane (Lemma \ref{nowall}), we see that $S$ is also $\sigma_{\alpha, \beta}$-stable for any $\alpha,\beta$. As the shift preserves stability, also $S[1]$ is both $\sigma_H$- and $\sigma_{\alpha, \beta}$-semistable.
We have:
\[
Z_H(S)=2+4\sqrt{-1}\,\qquad\text{and}\qquad 
Z_{\alpha, -\frac{1}{2}}(S[1]) = -2\alpha^2 - \frac{1}{2}
\]
which yields $(S[1])[-1]\in \mathrm{Coh}(Q_3)$ with
\[
\mu_H((S[1])[-1])=-\frac{1}{2}\leq -\frac{1}{2}; \qquad \mu_{\alpha,-\frac{1}{2}}(S[1])=+\infty
\]
and so, by Remark \ref{remark:four-cases}, $S[1]\in \mathrm{Coh}_{\alpha,-\frac{1}{2}}(Q_3)$.
\end{proof}
This suggests the following definition, of which the spinor bundle on $Q_3$ is an example.
\begin{definition}
Let $\sigma=(\mathcal{A},Z)$ be a numerical stability condition on a numerically finite triangulated category $\mathcal{D}$ with a Serre functor $\mathbb{S}$. A $\sigma$-semistable object $E$ in $\mathcal{D}$ is called a $d$-codimensional \emph{Bridgeland point object} for $\sigma$ if
\begin{enumerate}
    \item $\mathbb{S}(E)\cong E[d]$;
    \item the numerical class $[E]$ of $E$ is indecomposable in the effective numerical Cone $K^+_{\mathrm{num}}(\mathcal{D}_{\phi_E})$ of $\mathcal{D}_{\phi_E}$, i.e., in the image of $\mathcal{D}_{\phi_E}$ in $K_{\mathrm{num}}(\mathcal{D})$, where $\phi_E$ is the phase of $E$ with respect to the stability function $Z$, and $\mathcal{D}_{\phi_E}\subseteq \mathcal{D}$ is the associated slice.
\end{enumerate}
\end{definition}
\begin{lemma}\label{lemma:schur-type}
A $d$-codimensional Bridgeland point object is a $d$-codimensional Bondal-Orlov point object.
\end{lemma}
\begin{proof}
Let $E$ be a Bridgeland point object. The condition $\mathbb{S}(E)\cong E[d]$ is true by definition, while the Hom-vanishing $\mathrm{Hom}_{\mathcal{D}}(E,E[n])=0$ for every $n<0$ comes from the defining properties of the slicing $\{\mathcal{D}_\phi\}_{\phi\in \mathbb{R}}$ associated to the stability condition $(\mathcal{A},Z)$. To see that $\mathrm{Hom}_{\mathcal{D}}(E,E)\cong \mathbb{C}$ we use that the slice $\mathcal{D}_{\phi_E}$ is a full abelian subcategory of $\mathcal{D}$ and prove that 
\[
\mathrm{End}_{\mathcal{D}_{\phi_E}}(E)\cong \mathbb{C}
\]
by a Schur lemma-type argument.
The category $\mathcal{D}_{\phi_E}$ is abelian, hence for every morphism $f\colon E \to E$ we have a short exact sequence 
\[
0\to \mathrm{ker}_{\phi_E}(f)\to E \to E/\mathrm{ker}_{\phi_E}(f)\to 0
\]
in $\mathcal{D}_{\phi_E}$. From this we get
\[
[E]=[\mathrm{ker}_{\phi_E}(f)]+[E/\mathrm{ker}_{\phi_E}(f)]
\]
in $K^+_{\mathrm{num}}(\mathcal{D}_{\phi_E})$. Since in presence of a numerical stability condition a semistable object (i.e., an object in a slice $\mathcal{D}_\phi$ for some $\phi$) can not have zero numerical class unless it is the zero object, the fact that $E$ is a Bridgeland point object, and that the stability condition $\sigma$ is numerical, implies
$\mathrm{ker}_{\phi_E}(f)=0$ or $\mathrm{ker}_{\phi_E}(f)=E$. Arguing in the same way for the image of $f$ in  $\mathcal{D}_{\phi_E}$ one concludes that $f$ is either zero or an isomorphism, i.e., that $\mathrm{End}_{\mathcal{D}_{\phi_E}}(E)$ is a division ring. As we are working over the algebraically closed field $\mathbb{C}$, this implies $\mathrm{End}_{\mathcal{D}_{\phi_E}}(E)\cong \mathbb{C}$.
\end{proof}

\begin{remark}
Since $Z(K^+_{\mathrm{num}}(\mathcal{D}_{\phi_E}))\subseteq Z(K_{\mathrm{num}}(\mathcal{D}))\cap \mathbb{R}_{> 0}e^{\pi i\phi_E}$, one sees that if $Z(E)$ is indecomposable in $Z(K_{\mathrm{num}}(\mathcal{D})\cap \mathbb{R}_{> 0}e^{\pi i\phi_E}$, then surely $[E]$ is indecomposable in $K^+_{\mathrm{num}}(\mathcal{D}_{\phi_E})$. In concrete situations this often allows us to check the numerical indecomposability condition in the definition of Bridgeland point object via simple considerations on vectors in $\mathbb{C}$.
\end{remark}

\bigskip

The fact that $S$ is an exceptional object in ${\mathcal{D}^b}(Q_3)$ can be equivalently stated by saying that the Fourier-Mukai transform
\[
\Phi^S_{\mathcal{M}_S\to Q_3}\colon {\mathcal{D}^b}(\mathcal{M}_{S})\to {\mathcal{D}^b}(Q_3),
\]
where $\mathcal{M}_S$ is the one-point moduli space consisting of the equivalence class of the object $S$ alone, is fully faithful. While this way of formulating the exceptionality of $S$ is surely an overkill, it allows us to see the use we made of the property of $S$ of being a point object in $\mathrm{Ku}(Q_3)$ as a particular instance of the following Proposition, which is a reformulation of \cite[Theorem 5.1]{Bri98}.
\begin{prop}\label{prop:fully-faithful}
Let $X$ be an smooth projective variety, and let $\mathcal{K}$ be an admissible triangulated subcategory of ${\mathcal{D}^b}(X)$ endowed with a numerical stablity condition. Assume there exists a $d$-dimensional smooth projective variety which is a fine moduli space $\mathcal{M}_v$ of $d$-codimensional Bridgeland point objects in $\mathcal{A}_\mathcal{K}\subseteq \mathcal{K}$ of numerical class $v$ and let $\mathcal{E}\in {\mathcal{D}^b}(\mathcal{M}_v\times X)$ be the corresponding universal family. 
If for every two nonismorphic objects $E_{y_1}$ and $E_{y_2}$ parametrised by $\mathcal{M}_v$ one has
\begin{equation}\label{eq:orthogonality}
\mathrm{Hom}_{{\mathcal{D}^b}(X)}(E_{y_1},E_{y_2}[n])=0, \qquad\text{for $0<n<d$}, 
\end{equation}
then
\[
\Phi_{\mathcal{M}_v\to X}^\mathcal{E}\colon {\mathcal{D}^b}(\mathcal{M}_v)\to {\mathcal{D}^b}(X)
\]
is fully faithful.
\end{prop}
\begin{proof}
Let $E_y$ be the object of $\mathcal{K}$ corresponding to the point $y$ in $\mathcal{M}_v$. We have, by definition of universal family, that $E_y=\Phi_{\mathcal{M}_v\to X}^\mathcal{E}(\mathcal{O}_y)$. Therefore, since $\mathcal{K}$ is a full subcategory of ${\mathcal{D}^b}(X)$, in order to show that the assumptions of \cite[Theorem 5.1]{Bri98} are satisfied we need to show that
\begin{enumerate}
    \item for any point $y\in \mathcal{M}_v$ one has $\mathrm{Hom}_{\mathcal{K}}(E_y,E_y)=\mathbb{C}$;
    \item for any point $y\in \mathcal{M}_v$ one has $\mathrm{Hom}_{\mathcal{K}}(E_y,E_y[n])=0$ for any $n<0$ and any $n>d$;
    \item for any two distinct points $y_1,y_2\in \mathcal{M}_v$ one has $\mathrm{Hom}_{\mathcal{K}}(E_{y_1},E_{y_2}[n])=0$, for any $n\in \mathbb{Z}$
\end{enumerate}
Let $\mathcal{K}_\lambda\subseteq \mathcal{A}_{\mathcal{K}}$ be the slice of $\mathcal{K}$ corresponding to the numerical class $v$. As the slice $\mathcal{K}_\lambda$ is a full abelian subcategory of $\mathcal{K}$, we have
\[
\mathrm{Hom}_{\mathcal{K}}(E_{y_1},E_{y_2})=\begin{cases}
\mathbb{C}&\text{if $y_1=y_2$}\\
0&\text{if $y_1\neq y_2$}
\end{cases}
\]
by the same Schur's Lemma argument as in Lemma \ref{lemma:schur-type}. By the properties of slicings, one has  $\mathrm{Hom}_{\mathcal{K}}(E_{y_1},E_{y_2}[n])=0$, for any $y_1,y_2$ in $\mathcal{M}_v$ and any $n<0$. One then concludes by Serre duality in $\mathcal{K}$.
\end{proof}
\begin{remark}\label{rem:addington-thomas}
Notice that the orthogonality condition (\ref{eq:orthogonality}) is empty if $d\leq 1$. That is, derived categories of $d$-dimensional moduli spaces of Bridgeland stable $d$-codimensional point objects of $\mathcal{K}\subseteq {\mathcal{D}^b}(X)$ are automatically fully faithfully embedded in ${\mathcal{D}^b}(X)$ when $d=0,1$. Examples of this phenomenon are the 0-dimensional moduli space of the spinor bundle on the smooth quadric threefold $Q_3$ (consisting of a single point), and the 1-dimensional moduli space of spinor bundles on the intersection $Y_4$ of two quadrics in $\mathbb{P}^5$ (which is a genus 2 curve). For $d=2$ the orthogonality condition (\ref{eq:orthogonality}) reduces to the numerical condition $\chi(v,v)=0$. This is notably the case for the numerical classes of a cubic fourfold $W_4$ leading an equivalence between the Kuznetsov component $\mathrm{Ku}(W_4)$ and the derived category of a (possibly twisted) K3 surface \cite{addington-thomas,huybrechts,BLMNPS}
\end{remark}

\section{Stability conditions loathe phantomic summands}
We will now show that the existence of a stability condition on a numerically finite triangulated category $\mathcal{D}$ with a Serre functor automatically excludes the possibility that $\mathcal{D}$ has a nontrivial orthogonal decomposition $\mathcal{D}=\mathcal{D}_1\oplus\mathcal{D}_2$ such that one of the two summands is numerically trivial. This fact will be used to prove the fullness of the exceptional collection $(S,\mathcal{O}_{Q_3},\mathcal{O}_{Q_3}(1),\mathcal{O}_{Q_3}(2))$ in Section \ref{sec:fullness}. Most of the proofs in this Section are straightforward and can be skipped; we will only write them for the sake of completeness.

\bigskip 

Recall that a semiorthogonal decomposition $\mathcal{D}=\langle\mathcal{D}_1,\mathcal{D}_2\rangle$ is called \emph{orthogonal} if one has
\[
\mathrm{Hom}_{\mathcal{D}}(E_1,E_2)=0
\]
for any $E_1\in \mathcal{D}_1$ and $E_2\in \mathcal{D}_2$. When this happens we write 
\[ \mathcal{D} = \mathcal{D}_1 \oplus \mathcal{D}_2.  \]
The following is immediate from the definitions.
\begin{lemma}\label{lemma:orthogonal}
Let $\mathcal{A}\hookrightarrow \mathcal{B}$ be an admissible subcategory of the  triangulated  category $\mathcal{B}$. If $\mathcal{A}^\perp\subseteq {}^\perp\mathcal{A}$, then the semiorthogonal decomposition $\mathcal{B}=\langle \mathcal{A}^\perp,\mathcal{A}\rangle$ is an orthogonal decomposition.

\end{lemma}
An immediate application of \cite[Lemma 3.4]{Bri98} gives the following.
\begin{lemma}
Let $\mathcal{D} = \mathcal{D}_1 \oplus \mathcal{D}_2$
be a orthogonal decomposition of a triangulated category $\mathcal{D}$. Then every object $X$ in $\mathcal{D}$ can be decomposed as a biproduct $X=X_1\oplus X_2$, with $X_i\in \mathcal{D}_i$. Moreover, this decomposition is unique up to isomorphism. 
\end{lemma}

From this, we obtain that $t$-structures are well-behaved with respect to orhogoanl decompositions.
\begin{lemma}\label{lemma:t-structure-on-decomposition}
Let $\mathcal{D} = \mathcal{D}_1 \oplus \mathcal{D}_2$
be a orthogonal decomposition of a triangulated category $\mathcal{D}$, and let
$\langle \mathcal{D}_{\leq 0}, \mathcal{D}_{\geq 0} \rangle$ be a $t$-structure  on $\mathcal{D}$. Then 
\[ \mathcal{D}_{i;\leq 0} = \mathcal{D}_{\leq 0} \cap \mathcal{D}_i , \qquad \mathcal{D}_{i;\geq 0} = \mathcal{D}_{\geq 0} \cap \mathcal{D}_i
\]
is a $t$-structure on $\mathcal{D}_i$ for $i=1,2$. 
\end{lemma}
\begin{proof}

One easily sees that
\[
\mathcal{D}_{\leq 0}=\mathcal{D}_{1;\leq 0}\oplus \mathcal{D}_{2;\leq 0}.
\]
Namely, an object $X\in \mathcal{D}_{\leq 0}\subseteq \mathcal{D}$ can be written as $X=X_1\oplus X_2$ with $X_i\in \mathcal{D}_{i}$. For any $Y\in \mathcal{D}_{\geq 1}$ one has
\[
0=\mathrm{Hom}_{\mathcal{D}}(Y,X)=\mathrm{Hom}_{\mathcal{D}}(Y,X_1)\oplus \mathrm{Hom}_{\mathcal{D}}(Y,X_2),
\]
and so $X_i\in \mathcal{D}_{\leq 0}$ for $i=1,2$. Similarly, one shows $ \mathcal{D}_{\geq 0}=\mathcal{D}_{1;\geq 0}\oplus \mathcal{D}_{2;\geq 0}$.
Defining the truncation functor $\tau_{i;\leq0}\colon \mathcal{D}_i\to \mathcal{D}_{i;\leq 0}$ as the composition
\[
\mathcal{D}_i\hookrightarrow \mathcal{D}\xrightarrow{\tau_{\leq 0}} \mathcal{D}_{\leq 0} \xrightarrow{\pi_i}\colon \mathcal{D}_{i;\leq 0},
\]
where $\pi_i$ is the projection functor, and similarly for $\tau_{i;\geq0}$, one sees that $(\mathcal{D}_{i;\leq 0},\allowbreak \mathcal{D}_{i;\geq 0})$ is a $t$-structure on $\mathcal{D}_i$.

\end{proof}
\begin{remark}
In the setup of the above lemma, if $\mathcal{A}$ denotes the heart of $\langle \mathcal{D}_{\leq 0},\allowbreak \mathcal{D}_{\geq 0} \rangle$, one immediately sees that the heart of the induced $t$-structure $\langle \mathcal{D}_{i;\leq 0},\allowbreak \mathcal{D}_{i;\geq 0} \rangle$ is $\mathcal{A}_i=\mathcal{A}\cap \mathcal{D}_i$. Moreover one has $\mathcal{A}=\mathcal{A}_1\oplus \mathcal{A}_2$.
\end{remark}
\begin{lemma}\label{lemma:bounded}
In the same assumptions as in Lemma \ref{lemma:t-structure-on-decomposition}, if the t-structure $\langle \mathcal{D}_{\leq 0}, \allowbreak\mathcal{D}_{\geq 0} \rangle$ on $\mathcal{D}$ is bounded, then the same is true for the induced $t$-structure $\langle \mathcal{D}_{i;\leq 0},\allowbreak \mathcal{D}_{i;\geq 0} \rangle$ on $\mathcal{D}$, for $i=1,2$.
\end{lemma}
\begin{proof}
From the definition of the truncation functors $\tau_{i;\leq0}, \tau_{i\geq 0}$ in the proof of Lemma \ref{lemma:t-structure-on-decomposition}, and 
since the $t$-structure on $\mathcal{D}$ is bounded, for any $E_i\in \mathcal{D}_i\subseteq \mathcal{D}$, we have
\[
\tau_{i;\geq n}(E_i)=\tau_{\geq n}(E_i)= 0, \qquad \text{for $n>\!>0$}
\]
and
\[
\tau_{i;\leq n}(E_1)=\tau_{\leq n}(E_i)= 0, \qquad \text{for $n<\!<0$}
\]
which shows that the induced $t$-structure on $\mathcal{D}_i$ is bounded.

\[
\tau_{i;\geq n}(E_1)=\tau_{\geq n}(E_1)= 0, \qquad \text{for $n>\!>0$}
\]
and
\[
\tau_{i;\leq n}(E_1)=\tau_{\leq n}(E_1)= 0, \qquad \text{for $n<\!<0$}
\]
which shows that the induced $t$-structure on $\mathcal{D}_1$ is bounded.
\end{proof}
\begin{remark}
By the argument in Lemma \ref{lemma:bounded} we see that for any object $E_i$ in $\mathcal{D}_i$, for $i=1,2$, we have
\[
\mathcal{H}^n_{\mathcal{A}_i}(E_i)=\mathcal{H}^n_{\mathcal{A}}(E_i)
\]
for any $n\in \mathbb{Z}$, where $\mathcal{A}$ and $\mathcal{A}_i$ denote the hearts of the $t$-structures under consideration on $\mathcal{D}$ and $\mathcal{D}_i$, respectively, and $\mathcal{H}^n_{\mathcal{A}}\colon \mathcal{D}\to\mathcal{A}$ and $\mathcal{H}^n_{\mathcal{A}_i}\colon \mathcal{D}_i\to\mathcal{A}_i$ are the corresponding cohomology functors.
\end{remark}

\begin{definition}
Let $\mathcal{D}$ be a numerically bounded triangulated category with a Serre functor. We say that a triangulated subcategory $\mathcal{D}_{\mathrm{phan}}$ of $\mathcal{D}$ is a \emph{phantomic summand} if
\begin{enumerate}
    \item there exists an orthogonal decomposition $\mathcal{D}=\mathcal{D}_{\mathrm{body}}\oplus\mathcal{D}_{\mathrm{phan}}$;
    \item for every object $E$ in $\mathcal{D}_{\mathrm{phan}}$ one has $[E]=0$ in $K_{\mathrm{num}}(\mathcal{D})$. $K_{\mathrm{num}}(\mathcal{D}_{\mathrm{phan}})=0$.
\end{enumerate}
\end{definition}

\begin{lemma}\label{lemma:phantomic}
Let $\mathcal{D}$ be a numerically finite bounded triangulated category with a Serre functor with a numerical stability condition. Then $\mathcal{D}$ has no nonzero phantomic summands.
\end{lemma}
\begin{proof}
Let $\sigma=(\mathcal{A},Z)$ be a numerical stability condition on $\mathcal{D}$. Assume we have a phantomic summand
$\mathcal{D}_{\mathrm{phan}}$, and consider an orthogonal decomposition $\mathcal{D}=\mathcal{D}_{\mathrm{body}}\oplus\mathcal{D}_{\mathrm{phan}}$. By Lemmas \ref{lemma:t-structure-on-decomposition} and \ref{lemma:bounded}, $\mathcal{A}_{\mathrm{phan}}=\mathcal{A}\cap \mathcal{D}_{\mathrm{phan}}$ is the heart of a bounded $t$-structure on $\mathcal{D}_{\mathrm{phan}}$. If $E$ is a nonzero object in $\mathcal{A}_{\mathrm{phan}}$, then it is a nonzero object in $\mathcal{A}$ and so $Z(E)\neq 0$. On the other hand, if $E\in \mathcal{A}_{\mathrm{phan}}$, then $E\in\mathcal{D}_{\mathrm{phan}}$ and so $[E]=0$ in $K_{\mathrm{num}}(\mathcal{D}_{\mathrm{phan}})\subseteq K_{\mathrm{num}}(\mathcal{D})$. As the stability condition $\sigma$ is numerical, $Z\colon K(\mathcal{D})\to \mathbb{C}$ factors through $K_{\mathrm{num}}(\mathcal{D})$ and so $Z(E)=0$, a contradiction. This means that $\mathcal{A}_{\mathrm{phan}}=0$ and so $\mathcal{D}_{\mathrm{phan}}=0$.
\end{proof}

\begin{lemma}\label{lemma:i-is-equivalence}
Let $\mathcal{A}$ be and admissible subcategory of the triangulated category $\mathcal{B}$. If
\begin{enumerate}
\item $\mathcal{A}^\perp\subseteq {}^\perp\mathcal{A}$;
\item $\mathcal{B}$ is numerically finite and endowed with a Serre functor;
 \item $\mathcal{B}$ has a numerical stability condition;
    \item $K_{\mathrm{num}}(\mathcal{A})= K_{\mathrm{num}}(\mathcal{B})$;
   \end{enumerate}
then $\mathcal{A}=\mathcal{B}$.
\end{lemma}
\begin{proof}
 By Lemma \ref{lemma:orthogonal}, we have an orthogonal decomposition $
\mathcal{B}\cong \mathcal{A}^{\perp}\oplus \mathcal{A}$, and so a direct sum decomposition $K_{\mathrm{num}}(\mathcal{B})=K_{\mathrm{num}}(\mathcal{A})\oplus K_{\mathrm{num}}(\mathcal{A}^{\perp})$, by Remark \ref{rem:numerical-semiorthogonality}. Our assumptions give $K_{\mathrm{num}}(\mathcal{A}^{\perp})=0$ and so $\mathcal{A}^\perp$ is phantomic.
By Lemma \ref{lemma:phantomic}, $\mathcal{A}^\perp=0$ and so $\mathcal{B}=\mathcal{A}$.
\end{proof}


\section{Fullness of the standard exceptional collection on $Q_3$}\label{sec:fullness}

\begin{prop}\label{prop:collection-is-full}
The exceptional collection $(S,\mathcal{O}_{Q_3},\mathcal{O}_{Q_3}(1),\mathcal{O}_{Q_3}(2))$ of ${\mathcal{D}^b}(Q_3)$ is full.
\end{prop}
\begin{proof}
We have to show that $\langle S,\mathcal{O}_{Q_3},\mathcal{O}_{Q_3}(1),\mathcal{O}_{Q_3}(2)\rangle$ is a semiorthogonal decomposition of ${\mathcal{D}^b}(Q_3)$. By definition of ${\mathrm{Ku}}(Q_3)$, we have a semorthogonal decomposition $\langle {\mathrm{Ku}}(Q_3),\mathcal{O}_{Q_3},\allowbreak\mathcal{O}_{Q_3}(1),\mathcal{O}_{Q_3}(2)\rangle$, and we know from Lemma \ref{lemma:S-in-Ku} that $S\in \mathrm{Ku}(Q_3)$. Therefore all we have to show is that the inclusion $\langle S\rangle \hookrightarrow \mathrm{Ku}(Q_3)$ is an equivalence. Since $\langle S\rangle \hookrightarrow \mathrm{Ku}(Q_3)$ induces an isomorphism between $K_{\mathrm{num}}(\langle S\rangle)$ and $K_{\mathrm{num}}(\mathrm{Ku}(Q_3))$, and since we have a stability condition on $\mathrm{Ku}(Q_3)$ by Proposition \ref{prop:stability-condition-on-ku}, this is a particular case of Proposition \ref{prop:FM-equivalence} below, which is in turn a version of \cite[Theorem 5.4]{Bri98}.
\end{proof}

\begin{prop}\label{prop:FM-equivalence}
Let $X$ be an smooth projective variety, and let $\mathcal{K}$ be an admissible triangulated subcategory of ${\mathcal{D}^b}(X)$ endowed with a numerical stability condition. Assume there exists a $d$-dimensional smooth projective variety $\mathcal{M}_v$ which is a fine moduli space of $d$-codimensional Bridgeland point objects in $\mathcal{A}_{\mathcal{K}}\subseteq\mathcal{K}$ of numerical class $v$ \sout{phase $\lambda$} and let $\mathcal{E}\in {\mathcal{D}^b}(\mathcal{M}_v\times X)$ be the corresponding universal family. 
If for every two nonismorphic objects $E_{y_1}$ and $E_{y_2}$ parametrised by $\mathcal{M}_v$ one has
\begin{equation*}
\mathrm{Hom}_{{\mathcal{D}^b}(X)}(E_{y_1},E_{y_2}[n])=0, \qquad\text{for $0<n<d$},
\end{equation*}
then the Fourier-Mukai transform $\Phi^{\mathcal{E}}_{\mathcal{M}\to X}$ induces an equivalence $\Phi\colon {\mathcal{D}^b}(\mathcal{M}_v)\to \mathcal{K}$ if and only if it induces an isomorphism $\Phi_{\mathrm{num}}\colon K_{\mathrm{num}}(\mathcal{M}_v)\to K_{\mathrm{num}}(\mathcal{K})$.
\end{prop}
\begin{proof}
Since $\mathcal{M}_v$ is a moduli space of objects in $\mathcal{K}\subseteq {\mathcal{D}^b}(X)$, the Fourier-Mukai transform $\Phi^{\mathcal{E}}_{\mathcal{M}_v\to X}$ factors as
\[
\xymatrix{
{\mathcal{D}^b}(\mathcal{M}_v)\ar[r]^-{\Phi}\ar@/^1.5pc/[rr]^{\Phi^{\mathcal{E}}_{\mathcal{M}_v\to X}}&\mathcal{K}\ar[r]^-{j}&{\mathcal{D}^b}(X),
}
\]
where $j\colon \mathcal{K}\to {\mathcal{D}^b}(X)$ is the inclusion. Namely,  as $\mathcal{K}$ is admissible, to show that $\Phi^{\mathcal{E}}_{\mathcal{M}_v\to X}(F)\in \mathcal{K}$ for every $F\in \mathcal{D}^b(\mathcal{M}_v)$ we only need to show that 
\[
\mathrm{Hom}_{\mathcal{D}^b(X)}(G,\Phi^{\mathcal{E}}_{\mathcal{M}_v\to X}(F))=0
\]
for any $G\in {}^\perp\mathcal{K}$. The Fourier-Mukai transform $\Phi^{\mathcal{E}}_{\mathcal{M}_v\to X}$ has both a left and a right adjoint, see \cite[Lemma 4.5]{Bri98}, so
\[
\mathrm{Hom}_{\mathcal{D}^b(X)}(G,\Phi^{\mathcal{E}}_{\mathcal{M}_v\to X}(F))=
\mathrm{Hom}_{\mathcal{D}^b(\mathcal{M}_v)}((\Phi^{\mathcal{E}}_{\mathcal{M}_v\to X})^L(G),F)
\]
But $(\Phi^{\mathcal{E}}_{\mathcal{M}_v\to X})^L(G)=0$ for any $G\in {}^\perp\mathcal{K}$. Indeed, for any $y\in \mathcal{M}_v$ we have
\begin{align*}
\mathrm{Hom}_{\mathcal{D}^b(\mathcal{M}_v)}((\Phi^{\mathcal{E}}_{\mathcal{M}_v\to X})^L(G),\mathcal{O}_y)&= 
\mathrm{Hom}_{\mathcal{D}^b(X)}(G,\Phi^{\mathcal{E}}_{\mathcal{M}_v\to X}(\mathcal{O}_y))\\
&=\mathrm{Hom}_{\mathcal{D}^b(X)}(G,E_y)\\
&=0,
\end{align*}
since for every $y\in \mathcal{M}_v$ we have $E_y\in\mathcal{K}$. As $\{\mathcal{O}_y\}_{y\in \mathcal{M}_v}$ is a spanning class for $\mathcal{D}^b(\mathcal{M}_v)$ this implies $(\Phi^{\mathcal{E}}_{\mathcal{M}_v\to X})^L(G)=0$ and so $\Phi^{\mathcal{E}}_{\mathcal{M}_v\to X}(F)\in \mathcal{K}$.

Clearly, if $\Phi$ is an equivalence, then it induces an isomorphism between $K_{\mathrm{num}}(\mathcal{M}_v)$ and $K_{\mathrm{num}_v}(\mathcal{K})$. Vice versa, assume that $\Phi$ induces an isomorphism between $(K_{\mathrm{num}}(\mathcal{M}_v),\chi_{{\mathcal{D}^b}(\mathcal{M}_v)})$ and $K_{\mathrm{num}}(\mathcal{K}),\chi_{\mathcal{K}})$, and
denote by $\mathcal{F}$ the essential image of $\Phi^{\mathcal{E}}_{\mathcal{M}_v\to X}$ in $\mathcal{D}^b(X)$. As
$\Phi^{\mathcal{E}}_{\mathcal{M}_v\to X}$ is fully faithful by Proposition \ref{prop:fully-faithful}, we obtain the diagram
\[
\mathcal{D}^b(\mathcal{M}_v)\xrightarrow[\raisebox{4pt}{$\sim$}]{\Phi} \mathcal{F}\subseteq \mathcal{K}\subseteq \mathcal{D}^b(X)
\]
and we are reduced to showing that $\mathcal{F}=\mathcal{K}$. 
The subcategory $\mathcal{K}$ is admissible by assumption, and $\mathcal{F}$ is admissible as the Fourier-Mukai transform $\Phi^{\mathcal{E}}_{\mathcal{M}_v\to X}$ has both a left and a right adjoint, see \cite[Lemma 4.5]{Bri98}. Therefore $\mathcal{F}$ is an admissible subcategory of $\mathcal{K}$, by Remark \ref{rem:admissible-3}. Denote by ${}^\perp \mathcal{F}$ and by $\mathcal{F}^\perp$ the left and right orthogonal of $\mathcal{F}$ in $\mathcal{K}$, respectively. As we have a numerical stability condition on $\mathcal{K}$ and $K_{\mathrm{num}}(\mathcal{F})=K_{\mathrm{num}}(\mathcal{K})$, in order to apply Lemma \ref{lemma:i-is-equivalence} we only need to show that $\mathcal{F}^\perp\subseteq {}^\perp \mathcal{F}$. Let therefore $F$ be an object in $\mathcal{F}^\perp$. As $E_y$ is a $d$-codimensional point object in $\mathcal{K}$ for every $y\in \mathcal{M}_v$ we have, for every integer $k$ and for every point $y$ in $\mathcal{M}$,
\begin{align*}
\mathrm{Hom}_{\mathcal{K}}(F,E_y[k])=\mathrm{Hom}_{\mathcal{K}}(F,\mathbb{S}_{\mathcal{K}}E_y[k-d])
=\mathrm{Hom}_{\mathcal{K}}(E_y[k-d],F)^\vee=0.
\end{align*}
Therefore, 
\[\mathcal{F}^\perp
\subseteq {}^\perp\{E_y[k]\}_{y\in \mathcal{M}_v,k\in\mathbb{Z}}.
\]
Since $\{\mathcal{O}_y\}{y\in \mathcal{M}_v}$ is a spanning class for ${\mathcal{D}^b}(\mathcal{M}_v)$
and $\Phi\colon {\mathcal{D}^b}(\mathcal{M}_v)\xrightarrow{\sim} \mathcal{F}$ is an equivalence, $\{E_y\}_{y\in \mathcal{M}_v}=\{\Phi(\mathcal{O}_y\}_{y\in \mathcal{M}_v}$ is a spanning class for $\mathcal{F}$ and so \allowbreak
${}^\perp\{E_y[k]\}_{y\in \mathcal{M}_v,k\in\mathbb{Z}}={}^\perp\mathcal{F}$.
\end{proof}

\begin{remark}
Proposition \ref{prop:FM-equivalence} can be used to show that the Kuznetsov component $\mathrm{Ku}(Y_4)$ is equivalent to the derived category of a genus 2 curve. Namely, all the proofs given in the present article for the spinor bundle $S$ on $Q_3$ more or less verbatim apply to the spinor bundles on $Y_4$, to show that they are 1-codimensional Bridgeland stable objects in $\mathrm{Ku}(Y_4)$ with respect to the stability condition induced by $(\mathrm{Coh}(Y_4)_{\alpha,-\frac{1}{2}}^0,Z_{\alpha,-\frac{1}{2}}^0)$. The only additional check one needs, besides the properties of spinor bundles on quadrics recalled or derived above, is showing that the spinor bundles (and, more generally, spinor sheaves) on $Y_4$ are objects of the Kuznetsov component  $\mathrm{Ku}(Y_4)$. Since $\mathrm{Ku}(Y_4)$ is defined as the right orthogonal to the exceptional collection $(\mathcal{O},\mathcal{O}(1))$ in ${\mathcal{D}^b}(Y_4)$, an argument like that of Lemma \ref{lemma:S-in-Ku} proves that this is equivalent to showing that
$H^n(Y_4;S(-i))=0$ for $i\in\{0,1\}$ and every $n\in \mathbb{Z}$. By definition, $S$ is the restriction to $Y_4$ of a spinor bundle over a 4-dimensional quadric $Q_4$ containing $Y_4$ (more generally, of spinor sheaf when $Q_4$ is singular), and the required cohomological vanishing is then immediate from the short exact sequence
\[
0\to S_{Q_4}(-2) \to S_{Q_4}\to S\to 0
\]
and from the homological vanishing for $S_{Q_4}(-i)$.
The genus 2 curve arises as the moduli space of such sheaves. This is classical and due to Bondal-Orlov \cite{bondal-orlov}: for every point $p$ in the pencil of 4-dimensional quadrics containing $Y_4$ corresponding to a smooth quadric $Q_{4;p}$ one has two non isomorphic spinor bundles $S^+_{Q_{4,p}},S^-_{Q_{4,p}}$ whose restriction to $Y_4$ gives two nonisomorphic spinor bundles on $Y_4$. This realizes the moduli space of spinor bundles on $Y_4$ as a double cover of $\mathbb{P}^1$ ramified over the 6 points corresponding to singular quadrics in the pencil, i.e., as an hyperelliptic genus 2 curve. Results of Section \ref{section:point-objects} identify this moduli space with the moduli space of Bridgeland stable objects in $\mathrm{Coh}^0_{\alpha,-\frac{1}{2}}(Y_4)[-1]\cap \mathrm{Ku}(Y_4)$ with Chern character $2-H+\frac{1}{12}H^3$.
\end{remark}

\begin{remark}
The only missing ingredient in order to generalize Proposition \ref{prop:collection-is-full} to an arbitrary $n$-dimensional quadric hypersurface $Q_n$ is the existence of a numerical stability condition on the residual category of the exceptional collection $(\mathcal{O}_{Q_n},\mathcal{O}_{Q_n}(1),\dots,\mathcal{O}_{Q_n}(n-1))$ in ${\mathcal{D}^b}(Q_n)$. Of course, one knows from \cite{kapranov} that one has a full exceptional collection 
\begin{equation}\label{eq:Kapranov-even}
    (S^+,S^{-},\mathcal{O}_{Q_n},\mathcal{O}_{Q_n}(1),\dots,\mathcal{O}_{Q_n}(n-1))
\end{equation}
for even $n$ and 
\begin{equation}\label{eq:Kapranov-odd}
    (S,\mathcal{O}_{Q_n},\mathcal{O}_{Q_n}(1),\dots,\mathcal{O}_{Q_n}(n-1))
\end{equation}
for odd $n$, so that $\mathrm{Ku}(Q_n)$ is equivalent to the derived category of one or of two points depending on the parity of $n$, and so it surely has numerical stability conditions. But if one were able to prove a priori, i.e., without using the knowledge of the full excepional collections (\ref{eq:Kapranov-even}-\ref{eq:Kapranov-odd}), that there existed a numerical stability condition on $\mathrm{Ku}(Q_n)$, one could use the constructions in this article to rediscover the full exceptional collections (\ref{eq:Kapranov-even}-\ref{eq:Kapranov-odd}). 
\end{remark}

\begin{remark}
Fano threefolds of Picard rank 1 are known to have stability conditions (\cite{li:stability-conditions}, preceeded by \cite{macri12} for the particular case of $\mathbb{P}^3$ and by \cite{schmidt13} for the case of the quadric threefold $Q_3$ considered in this article). We preferred not to use this stronger result and only use the weak stability condition (\ref{eq:weak-stability-function}) on ${\mathcal{D}^b}(Q_3)$ to induce a stability condition on $\mathrm{Ku}(Q_3)$ following \cite{BLMS}, as only a stability condition on the residual category is relevant to the constructions in this article, and weak stability conditions on the ambient category are much easier to have with respect to actual stability conditions. For instance, (\ref{eq:weak-stability-function}) defines a weak stability condition on every smooth projective variety, regardless of its dimension. It is therefore reasonable to conjecture that, in principle, proving the existence of a numerical stability condition  on a residual category of the form $\mathrm{Ku}(X)$ should be possible in many cases, even without the presence of a true stability condition on ${\mathcal{D}^b}(X)$. Hopefully, this should apply for instance to the residual components $\mathrm{Ku}(Q_n)$ of smooth quadrics.
\end{remark}

\begin{remark}
Another possible application of Proposition \ref{prop:FM-equivalence} is the follwing. Let $W_4$ be a smooth cubic fourfold, and let $v$ be an indecomposable numerical class in $K_{\mathrm{num}}^+(W_4)$ with $\chi_{W_4}(v,v)=0$. Let $\sigma$ be a numerical stability condition on $\mathrm{Ku}(W_4)$. Since $\mathrm{Ku}(W_4)$ is a 2-Calabi-Yau category, i.e., $\mathbb{S}_{\mathrm{Ku}(W_4)}=[2]$, see, e.g., \cite{Kuz15}, we see that any $\sigma$-semistable object in $\mathcal{A}_{\mathrm{Ku}(W_4)}$ is a codimension 2 Bridgeland point object in $\mathrm{Ku}(W_4)$. If a fine moduli space $\mathcal{M}_v$ for $\sigma$-semistable object in $\mathcal{A}_{\mathrm{Ku}(W_4)}$ exists which is a projective variety, it will be  2-dimensional. Indeed, the numerical condition $\chi_{W_4}(v,v)=0$ and Serre duality for $\mathrm{Ku}(W_4)$ give
\[
\mathrm{hom}(E,E[1])=\mathrm{hom}(E,E)+\mathrm{hom}(E,E[2])=2 \mathrm{hom}(E,E)=2,
\]
for any Bridgeland point object of numerical class $v$, where we wrote $\mathrm{hom}(F,G)$ for the dimension of the hom-space $\mathrm{Hom}_{\mathrm{Ku}(W_4)}(F,G)$. So, by Remark \ref{rem:addington-thomas} and Proposition \ref{prop:FM-equivalence}, if such a  $\mathcal{M}_v$ exists then we have an equivalence
\[
\mathcal{D}^b(\mathcal{M}_v)\cong \mathrm{Ku}(W_4).
\]
In particular, $\mathcal{M}_v$ will be a K3 surface. See \cite{addington-thomas,huybrechts,BLMNPS} for additional information on this result.
\end{remark}

\end{document}